\documentclass[12pt]{article}

\usepackage[a4paper, margin=2.6cm]{geometry}
\usepackage[english]{babel}
\usepackage{amssymb, amsmath, amsthm, mathtools}
\usepackage{graphicx, cite, xcolor, enumitem}
\usepackage{longtable, pdflscape, booktabs, caption, multicol}

\usepackage[hyphens]{url}
\usepackage[colorlinks=true,citecolor=black,linkcolor=black,urlcolor=blue]{hyperref}
\usepackage[hyphenbreaks]{breakurl}

\theoremstyle{plain}
\newtheorem{theorem}{Theorem}[section]
\newtheorem{lemma}[theorem]{Lemma}
\newtheorem{corollary}[theorem]{Corollary}
\newtheorem{proposition}[theorem]{Proposition}

\newcommand{\arxiv}[2]{\href{https://arxiv.org/abs/#1}{\texttt{arXiv:#1}} \texttt{[#2]}}
\newcommand{\doi}[1]{\url{https://doi.org/#1}}

\usepackage{listings}
\definecolor{mauve}{rgb}{0.58,0,0.82}
\definecolor{dkgreen}{rgb}{0,0.6,0}
\lstdefinestyle{pitonche} {
    language = Python,
    basicstyle = footnotesizettfamily,
    showspaces = false,
    showstringspaces = false,
    breakautoindent = true,
    flexiblecolumns = true,
    keepspaces = true,
    stepnumber = 1,
    xleftmargin = 0pt
}
\usepackage{authblk}

\title{On the maximum spectral radius of connected\\ graphs with a prescribed order and size}

\author[1,2]{Ivan Damnjanović}
\affil[1]{Faculty of Electronic Engineering, University of Niš, Niš, Serbia}
\affil[2]{Faculty of Mathematics, Natural Sciences and Information Technologies, University of Primorska, Koper, Slovenia}

\date{}

\begin{document}

\maketitle

\begin{abstract}
The spectral radius of a graph is the largest modulus of an eigenvalue of its adjacency matrix. Let $\mathcal{C}_{n, e}$ be the set of all the connected simple graphs with $n$ vertices and $n - 1 + e$ edges. Here, we solve the spectral radius maximization problem on $\mathcal{C}_{n, e}$ when $e \le 130$ or $n \ge e + 2 + 13\sqrt{e}$.
\end{abstract}

\bigskip\noindent
{\bf Keywords:} spectral radius, index, connected graph, threshold graph, stepwise matrix.

\bigskip\noindent
{\bf Mathematics Subject Classification:} 05C50, 05C35.

\section{Introduction}

Throughout the paper, we take all the graphs to be undirected, finite and simple. For a given graph $G$, we use $V(G)$ and $A(G)$ to signify its vertex set and adjacency matrix, respectively. The \emph{spectral radius}, also known as the \emph{index}, of a graph $G$, denoted by $\rho(G)$, is the largest modulus of an eigenvalue of $A(G)$. It is known that $\rho(G)$ is an eigenvalue of $A(G)$ for any graph $G$; see \cite[Chapter 8]{godroy2001}. For more results on the spectral radii of graphs, the reader is referred to \cite{stevanovic2018} and the references therein.

We use $G_1 \lor G_2$ and $G_1 + G_2$ to signify the join and disjoint union, respectively, of graphs $G_1$ and $G_2$. A graph $G$ is a \emph{threshold graph} if its vertices can be ordered in such a way that its corresponding adjacency matrix is stepwise in the sense that whenever $A_{ij} = 1$ with $1 \le i < j \le n$, we have $A_{k \ell} = 1$ for any $k, \ell \in \{1, 2, \ldots, n \}$ such that $k < \ell$, $k \le i$ and $\ell \le j$; see \cite{chvaham1977, mahpel1995}.

Let $\mathcal{G}_m$ be the set of all the graphs with $m \in \mathbb{N}_0$ edges. In 1976, Brualdi and Hoffman \cite[p.\ 438]{berfourversot1978} initiated the extremal problem of maximizing the spectral radius on $\mathcal{G}_m$. Afterwards, Brualdi and Hoffman obtained the next two results in 1985.

\begin{theorem}[\hspace{1sp}{\cite[Theorem~2.1]{bruhoff1985}}]
Suppose that $G$ attains the maximum spectral radius on $\mathcal{G}_m$ for some $m \in \mathbb{N}_0$. Then $G$ is a threshold graph.
\end{theorem}

\begin{theorem}[\hspace{1sp}{\cite[Theorem~2.2]{bruhoff1985}}]\label{bruhoff_th}
Let $m = \binom{k}{2}$ with $k \in \mathbb{N}$. Then graph $G$ attains the maximum spectral radius on $\mathcal{G}_m$ if and only if $G \cong K_k + rK_1$ for some $r \in \mathbb{N}_0$.
\end{theorem}

Further work on the spectral radius maximization problem on $\mathcal{G}_m$ was done by Friedland \cite{friedland1985, friedland1988} and Stanley \cite{stanley1987}. In a subsequent paper from 1988, Rowlinson finalized the solution to the problem as follows.

\begin{theorem}[\hspace{1sp}{\cite{rowlinson1988}}]\label{row_th}
    Let $m = \binom{k}{2} + t$ with $k \in \mathbb{N}$ and $t \in \{1, 2, \ldots, k - 1 \}$. Then graph $G$ attains the maximum spectral radius on $\mathcal{G}_m$ if and only if $G \cong \left(K_t \lor \left(K_{k - t} + K_1\right)\right) + rK_1$ for some $r \in \mathbb{N}_0$.
\end{theorem}

For any $e \in \mathbb{N}_0$, let $k_e$ be the largest natural number such that $\binom{k_e}{2} \le e$ and let $t_e := e - \binom{k_e}{2}$. Note that $0 \le t_e \le k_e - 1$. We observe that for any $n \in \mathbb{N}$ and $e \in \mathbb{N}_0$, there exists a connected graph of order $n$ and size $n - 1 + e$ if and only if $n \ge b_e$, where
\[
    b_e := \begin{cases}
        1, & e = 0,\\
        k_e + 1, & t_e = 0 \mbox{ and } e \neq 0,\\
        k_e + 2, & t_e \neq 0
    \end{cases} \qquad (e \in \mathbb{N}_0) .
\]
Let $\mathcal{C}_{n, e}$ be the set comprising the connected graphs of order $n$ and size $n - 1 + e$. Now, for each $e \in \mathbb{N}_0$ and $n \ge b_e$, let $\mathcal{D}_{n, e} \in \mathcal{C}_{n, e}$ be defined as
\[
    \mathcal{D}_{n, e} := \begin{cases}
        \left( \left( \left( K_{k_e - t_e} + K_1 \right) \lor K_{t_e} \right) + (n - k_e - 2) K_1 \right) \lor K_1, & n \ge k_e + 2,\\
        K_n, & n < k_e + 2 .
    \end{cases}
\]
Also, for each $e \in \mathbb{N}_0$ and $n \ge e + 2$, let $\mathcal{V}_{n, e} \in \mathcal{C}_{n, e}$ be defined as
\[
    \mathcal{V}_{n, e} := \left( K_{1, e} + (n - e - 2) K_1 \right) \lor K_1.
\]

In 1986, the spectral radius maximization problem on $\mathcal{C}_{n, e}$ was initiated  by Brualdi and Solheid, who obtained the next three results.

\begin{theorem}[\hspace{1sp}{\cite[Theorem~2.1]{brusol1986}}]\label{con_thresh_th}
    Let $e \in \mathbb{N}_0$ and $n \ge b_e$, and suppose that $G$ attains the maximum spectral radius on $\mathcal{C}_{n, e}$. Then $G$ is a threshold graph.
\end{theorem}

\begin{theorem}[\hspace{1sp}{\cite[Theorems~3.1 and 3.2]{brusol1986}}]\label{brusol_d_th}
    Let $e \in \{0, 1, 2, 3\}$ and $n \ge b_e$. Then $\mathcal{D}_{n, e}$ is the unique graph attaining the maximum spectral radius on $\mathcal{C}_{n, e}$.
\end{theorem}

\begin{theorem}[\hspace{1sp}{\cite[Theorem~3.3]{brusol1986}}]\label{brusol_v_th}
    For each $e \in \{ 4, 5, 6 \}$, there is an $f(e)$ such that for every $n \ge f(e)$, the graph $\mathcal{V}_{n, e}$ is the unique graph attaining the maximum spectral radius on $\mathcal{C}_{n, e}$.
\end{theorem}

Theorem \ref{con_thresh_th} was later proved by Simić, Li Marzi and Belardo \cite{silimabe2004} with a different strategy. For alternative proofs of Theorem \ref{brusol_d_th} for the case $e = 0$, see \cite{cosi1957} and \cite{lope1973}. In 1988, Theorem \ref{brusol_v_th} was generalized by Cvetkovi\'c and Rowlinson as follows.

\begin{theorem}[\hspace{1sp}{\cite{cvetrow1988}}]\label{cvetrow_v_th}
    For each $e \ge 4$, there is an $f(e)$ such that for every $n \ge f(e)$, the graph $\mathcal{V}_{n, e}$ is the unique graph attaining the maximum spectral radius on $\mathcal{C}_{n, e}$.
\end{theorem}

The value $f(e)$ from Theorem \ref{cvetrow_v_th} is not trivial to compute and it is larger than $e^2 (e + 2)^2$ for $e \ge 7$. In 1991, Bell solved the spectral radius maximization problem on $\mathcal{C}_{n, e}$ for the case $t_e = 0$.

\begin{theorem}[\hspace{1sp}{\cite{bell1991}}]\label{bell_th}
    For any $\lambda > 3$, let
    \[
        f(\lambda) = \frac{1}{2} (\lambda + 1)(\lambda + 6) + 7 + \frac{32}{\lambda - 3} + \frac{16}{(\lambda - 3)^2}.
    \]
    Then for any $e \ge 6$ such that $t_e = 0$, we have:
    \begin{enumerate}[label=\textbf{(\roman*)}]
        \item if $b_e \le n < f(k_e)$, then $\mathcal{D}_{n, e}$ is the unique graph attaining the maximum spectral radius on $\mathcal{C}_{n, e}$;
        \item if $n = f(k_e)$, then $\mathcal{D}_{n, e}$ and $\mathcal{V}_{n, e}$ are the only two graphs attaining the maximum spectral radius on $\mathcal{C}_{n, e}$;
        \item if $n > f(k_e)$, then $\mathcal{V}_{n, e}$ is the unique graph attaining the maximum spectral radius on $\mathcal{C}_{n, e}$.
    \end{enumerate}
\end{theorem}

In 2002, Olesky, Roy and van den Driessche resolved another case of the extremal problem.

\begin{theorem}[\hspace{1sp}{\cite{oleroydri2002}}]
    Let $e \ge 9$ be such that $t_e = k_e - 1$ and let $b_e + 1 \le n \le e - 1$. Then $\mathcal{D}_{n, e}$ is the unique graph attaining the maximum spectral radius on $\mathcal{C}_{n, e}$.
\end{theorem}

Bearing in mind all the existing results, it is natural to expect that no graph besides $\mathcal{D}_{n, e}$ and $\mathcal{V}_{n, e}$ (if defined) can be a solution to the spectral radius maximization problem on $\mathcal{C}_{n, e}$, for any $e \in \mathbb{N}_0$ and $n \ge b_e$. Here, we investigate this extremal problem and resolve some new cases. To begin, we introduce the auxiliary notion of T-subgraph and use it to construct a computer-assisted proof of the following proposition.

\begin{proposition}\label{computer_prop}
    Let $e \in \{4, 5, \ldots, 130 \}$ and $n \ge b_e$. Suppose that $G$ attains the maximum spectral radius on $\mathcal{C}_{n, e}$. Then $G \cong \mathcal{D}_{n, e}$ or ($n \ge e + 2$ and $G \cong \mathcal{V}_{n, e}$).
\end{proposition}

The computer-assisted proof is carried out through a \texttt{SageMath} script that can be found in \cite{GitHub}. Afterwards, we compare the graphs $\mathcal{D}_{n, e}$ and $\mathcal{V}_{n, e}$ and give a criterion on $n$ and $e$ that determines which of these two graphs has a larger spectral radius. This result yields the complete solution to the extremal problem for the case $e \le 130$. Finally, we apply a strategy inspired by Bell \cite{bell1991} to show that $\mathcal{V}_{n, e}$ is the unique solution when $e > 130$ and $n \ge e + 2 + 13\sqrt{e}$, thereby improving the bound from Theorem \ref{cvetrow_v_th}.

\section{Preliminaries}

We consider all spectral properties of a graph to correspond to its adjacency matrix. For any threshold graph $G$, we assume that $V(G) = \{1, 2, \ldots, n \}$, where $n = |V(G)|$, with the vertices being arranged in such a way that $A(G)$ is a stepwise matrix. Note that the stepwise adjacency matrix $A(G)$ is unique for each threshold graph $G$. For any $p \in \mathbb{N}_0$, we denote the identity matrix of order $p$ by $I_p$, and for any $p_1, p_2 \in \mathbb{N}_0$, we denote the zero (resp.\ all-ones) matrix with $p_1$ rows and $p_2$ columns by $O_{p_1, p_2}$ (resp.\ $J_{p_1, p_2}$). When the matrix size is clear from the context, we may drop the subscripts and write $I$, $O$ or $J$ for short.

We take all the polynomials and rational functions to be in the variable $\lambda$. We use $\mathcal{P}_G(\lambda)$ and $\mathcal{P}_A(\lambda)$ (resp.\ $\rho(G)$ and $\rho(A)$) to denote the characteristic polynomial (resp.\ spectral radius) of a graph $G$ and square matrix $A$, respectively. Also, we use $\rho(P)$ and $\rho_2(P)$ to denote the largest and second largest root of a polynomial $P(\lambda) \not\equiv 0$, taking into account root multiplicity. For convenience, we let $\rho(P) = -\infty$ if $P(\lambda)$ has no real root, and $\rho_2(P) = -\infty$ if $P(\lambda)$ has at most one (simple) real root.

Throughout the rest of the paper, for the sake of brevity, we write $k := k_e$, $t := t_e$ and $b := b_e$ when the value $e$ is fixed. We follow Bell \cite{bell1991} in denoting the adjacency matrix of $\mathcal{D}_{n, e}$ by $B$. Also, let $y$ be the corresponding Perron vector of $\mathcal{D}_{n, e}$ and let $\gamma = \rho(\mathcal{D}_{n, e})$. It is not difficult to observe that $y_1 \ge y_2 \ge \cdots \ge y_n > 0$; see \cite[Lemma~1]{rowlinson1988}. For each $e \in \mathbb{N}_0$ and $n \ge e + 2$, we denote the adjacency matrix of $\mathcal{V}_{n, e}$ by $C$. Besides, let $z$ be the corresponding Perron vector of $\mathcal{V}_{n, e}$ and let $\chi = \rho(\mathcal{V}_{n, e})$. Similarly, note that $z_1 \ge z_2 \ge \cdots \ge z_n > 0$.

\begin{lemma}\label{gnm_lemma}
    For any $e \ge 4$ such that $t_e \ge 1$, and any $n \ge b_e$, we have $y_2 = y_3 = \cdots = y_{t + 1}$, $y_{t + 2} = y_{t + 3} = \cdots = y_{k + 1}$ and $y_{k + 3} = y_{k + 4} = \cdots = y_n$, alongside
    \begin{align}
        \label{gnm_local_1} (\gamma + 1) y_1 &= y_1 + t y_2 + (k - t) y_{k + 1} + y_{k + 2} + (n - k - 2) y_n,\\
        \label{gnm_local_2} (\gamma + 1) y_2 &= y_1 + t y_2 + (k - t) y_{k + 1} + y_{k + 2},\\
        \label{gnm_local_3} (\gamma + 1) y_{k + 1} &= y_1 + t y_2 + (k - t) y_{k + 1},\\
        \label{gnm_local_4} \gamma y_{k + 2} &= y_1 + t y_2,\\
        \label{gnm_local_5} \gamma y_n &= y_1 \qquad \mbox{(provided $n > k_e + 2$)}.
    \end{align}
    Also, $\gamma$ is the largest real root of
    \begin{align}
    \label{comp_aux_2}
    \begin{split} 
        \lambda(\lambda + 1)(\lambda^3 &- \lambda^2(k - 1) - \lambda(k + t + 1) + (t + 1)(k - t - 1))\\
        &- (n - k - 2)(\lambda^3 - \lambda^2(k - 2) - \lambda(k + t - 1) + t(k - t - 1)) .
    \end{split}
    \end{align}
\end{lemma}
\begin{proof}
    From the eigenvalue--eigenvector equations for $\gamma$ and $y$, we directly obtain $y_2 = y_3 = \cdots = y_{t + 1}$, $y_{t + 2} = y_{t + 3} = \cdots = y_{k + 1}$ and $y_{k + 3} = y_{k + 4} = \cdots = y_n$, as well as \eqref{gnm_local_1}--\eqref{gnm_local_5}. Recall that $b_e = k_e + 2$ and suppose that $n \ge b_e + 1$. In this case, the vector
    \[
        \begin{bmatrix} y_1 & y_2 & y_{k + 1} & y_{k + 2} & y_n \end{bmatrix}^\intercal
    \]
    is an eigenvector of
    \begin{equation}\label{comp_aux_1}
        \begin{bmatrix}
            0 & t & k - t & 1 & n - k - 2\\
            1 & t - 1 & k - t & 1 & 0\\
            1 & t & k - t - 1 & 0 & 0\\
            1 & t & 0 & 0 & 0\\
            1 & 0 & 0 & 0 & 0
        \end{bmatrix}
    \end{equation}
    for the eigenvalue $\gamma$. Therefore, $\gamma$ is the spectral radius of \eqref{comp_aux_1}, hence a routine computation yields that $\gamma$ is also the largest real root of \eqref{comp_aux_2}. It is not difficult to verify that this claim also holds for the case $n = b_e$.
\end{proof}

The following lemma can be proved analogously to Lemma \ref{gnm_lemma}, so we omit its proof.

\begin{lemma}\label{hnm_lemma}
    For any $e \ge 4$ and $n \ge e + 2$, we have $z_3 = z_4 = \cdots = z_{e + 2}$ and $z_{e + 3} = z_{e + 4} = \cdots = z_n$, alongside
    \begin{align}
        \nonumber (\chi + 1) z_1 &= z_1 + z_2 + e z_3 + (n - e - 2) z_n,\\
        \label{hnm_local_2} (\chi + 1) z_2 &= z_1 + z_2 + e z_3,\\
        \label{hnm_local_3} \chi z_3 &= z_1 + z_2,\\
        \nonumber \chi z_n &= z_1 \qquad \mbox{(provided $n > e + 2$)}.
    \end{align}
    Also, $\chi$ is the largest real root of
    \[
        \lambda(\lambda + 1)(\lambda^2 - \lambda - 2e) - (n - e - 2)(\lambda^2 - e) .
    \]
\end{lemma}

We end the section with another result that can easily be verified.
    
\begin{lemma}\label{third_lemma}
    For any $e \in \mathbb{N}$, we have
    \[
        \rho(K_2 \lor eK_1) = \frac{1 + \sqrt{8e + 1}}{2} .
    \]
\end{lemma}

\section{T-subgraphs}\label{sc_tsub}

In the present section, we introduce the notion of T-subgraph and obtain several auxiliary lemmas. These results lead to a computer-assisted proof of Proposition \ref{computer_prop}. We begin as follows. For any $e \ge 1$, $n \ge b_e$ and threshold graph $G \in \mathcal{C}_{n, e}$, let $n - 1 = d_1 \ge d_2 \ge d_3 \ge \cdots \ge d_n$ be the degree sequence of $G$. Also, let $s$ be maximal such that $d_s \ge 2$. Then the \emph{T-subgraph} of graph $G$, denoted by $\mathcal{T}(G)$, is the subgraph of $G$ induced by the vertex set $\{2, 3, \ldots, s\}$. Note that $\mathcal{T}(G)$ is a connected threshold graph of order $s - 1$ and size $e$, and also observe that the threshold graphs from $\mathcal{C}_{n, e}$ can be uniquely determined by their T-subgraph. For the sake of brevity, let $\mathcal{T}_1(G) = \mathcal{T}(G) \lor K_1$. We proceed with the next lemma.

\begin{lemma}\label{poly_2_lemma}
    For any $e \ge 1$, $n \ge b_e$ and threshold graph $G \in \mathcal{C}_{n, e}$, we have that $\rho(G)$ is the largest real root of
    \[
        \lambda \, \mathcal{P}_{\mathcal{T}_1(G)}(\lambda) - (n - n') \, \mathcal{P}_{\mathcal{T}(G)}(\lambda) ,
    \]
    where $n' = |\mathcal{T}_1(G)|$.
\end{lemma}
\begin{proof}
    Let $x$ be the Perron vector of $G$, and observe that
    \[
        A(G) = \begin{bmatrix}
            0 & J_{1 \times (n' - 1)} & J_{1 \times (n - n')}\\
            J_{(n' - 1) \times 1} & A(\mathcal{T}(G)) & O\\
            J_{(n - n') \times 1} & O & O
        \end{bmatrix} .
    \]
    If $n > n'$, then we have $x_{n' + 1} = x_{n' + 3} = \cdots = x_n$, while the vector
    \[
        \begin{bmatrix}
            x_1 & x_2 & \cdots & x_{n'} & x_n
        \end{bmatrix}^\intercal
    \]
    is an eigenvector of
    \begin{equation}\label{tsub_aux_1}
        W = \begin{bmatrix}
            0 & J_{1 \times (n' - 1)} & n - n'\\
            J_{(n' - 1) \times 1} & A(\mathcal{T}(G)) & O\\
            1 & O & 0
        \end{bmatrix}
    \end{equation}
    for the eigenvalue $\rho(G)$. Therefore, $\rho(G) = \rho(W)$. We note that $\rho(G) = \rho(W)$ holds even when $n = n'$, in which case the last column of \eqref{tsub_aux_1} contains only zeros. To finalize the proof, it suffices to compute that
    \begin{align*}
        \pushQED{\qed}
        \mathcal{P}_W(\lambda) &= \begin{vmatrix}
            \lambda & -J & -(n - n')\\
            -J & \lambda I - A(\mathcal{T}(G)) & O\\
            -1 & O & \lambda
        \end{vmatrix}\\
        &= \begin{vmatrix}
            \lambda & -J & 0\\
            -J & \lambda I - A(\mathcal{T}(G)) & O\\
            -1 & O & \lambda
        \end{vmatrix} + \begin{vmatrix}
            0 & O & -(n - n')\\
            -J & \lambda I - A(\mathcal{T}(G)) & O\\
            -1 & O & \lambda
        \end{vmatrix}\\
        &= \lambda \begin{vmatrix}
            \lambda & -J\\
            -J & \lambda I - A(\mathcal{T}(G)) 
        \end{vmatrix} - (n - n') \begin{vmatrix}
            \lambda I - A(\mathcal{T}(G))
        \end{vmatrix}\\
        &= \lambda \, \mathcal{P}_{\mathcal{T}_1(G)}(\lambda) - (n - n') \, \mathcal{P}_{\mathcal{T}(G)}(\lambda) . \qedhere
    \end{align*}
\end{proof}

In view of Lemma \ref{poly_2_lemma}, we introduce the auxiliary function $\mathcal{R}_G \colon [\rho(\mathcal{T}_1(G)), +\infty) \to \mathbb{R}$ given by
\[
     \mathcal{R}_G(\lambda) = \frac{\lambda \, \mathcal{P}_{\mathcal{T}_1(G)}(\lambda)}{\mathcal{P}_{\mathcal{T}(G)}(\lambda)} \qquad (\lambda \ge \rho(\mathcal{T}(G) \lor K_1)),
\]
for any $e \ge 1$, $n \ge b_e$ and threshold graph $G \in \mathcal{C}_{n, e}$. From the theory of nonnegative matrices (see, e.g., \cite[Chapter 8]{godroy2001}), we have $\rho(\mathcal{T}_1(G)) > \rho(\mathcal{T}(G))$, which means that $\mathcal{R}_G$ is well-defined. As it turns out, $\mathcal{R}_G$ is monotone.

\begin{lemma}\label{rat_lemma}
    For any $e \ge 1$, $n \ge b_e$ and threshold graph $G \in \mathcal{C}_{n, e}$, the rational function $\mathcal{R}_G$ is a strictly increasing bijection from $[\rho(\mathcal{T}_1(G)), +\infty)$ onto $[0, +\infty)$.
\end{lemma}
\begin{proof}
    Let $n' = |\mathcal{T}_1(G)|$. From the theory of symmetric matrices, we have
    \[
        \mathcal{P}_{\mathcal{T}_1(G)}(\lambda) = (\lambda - \xi_1)(\lambda - \xi_2) \cdots (\lambda - \xi_{n'})
    \]
    and
    \[
        \mathcal{P}_{\mathcal{T}(G)}(\lambda) = (\lambda - \zeta_1)(\lambda - \zeta_2) \cdots (\lambda - \zeta_{n' - 1})
    \]
    for some real numbers $\xi_1 \le \xi_2 \le \cdots \le \xi_{n'}$ and $\zeta_1 \le \zeta_2 \le \cdots \le \zeta_{n' - 1}$. For any $\lambda > \xi_{n'}$, we get
    \begin{align*}
        \mathcal{R}'_G(\lambda) &= \frac{\left(\lambda \, \mathcal{P}_{\mathcal{T}_1(G)}(\lambda)\right)' \mathcal{P}_{\mathcal{T}(G)}(\lambda) - \lambda \, \mathcal{P}_{\mathcal{T}_1(G)}(\lambda) \, \mathcal{P}'_{\mathcal{T}(G)}(\lambda)}{\left(\mathcal{P}_{\mathcal{T}(G)}(\lambda)\right)^2}\\
        &= \frac{\lambda \, \mathcal{P}_{\mathcal{T}_1(G)}(\lambda) \, \mathcal{P}_{\mathcal{T}(G)}(\lambda) \left(\frac{1}{\lambda} + \sum_{i = 1}^{n'} \frac{1}{\lambda - \xi_i}\right) - \lambda \, \mathcal{P}_{\mathcal{T}_1(G)}(\lambda) \, \mathcal{P}_{\mathcal{T}(G)}(\lambda) \sum_{i = 1}^{n' - 1} \frac{1}{\lambda - \zeta_i}}{\left(\mathcal{P}_{\mathcal{T}(G)}(\lambda)\right)^2}\\
        &= \mathcal{R}_G(\lambda) \left( \frac{1}{\lambda} + \sum_{i = 1}^{n'} \frac{1}{\lambda - \xi_i} - \sum_{i = 1}^{n' - 1} \frac{1}{\lambda - \zeta_i} \right) .
    \end{align*}
    We trivially observe that $\mathcal{R}_G(\lambda) > 0$ for every $\lambda > \xi_{n'}$. Besides, we have $\xi_1 \le \zeta_1 \le \xi_2 \le \zeta_2 \le \cdots \le \zeta_{n' - 1} \le \xi_{n'}$ from the eigenvalue interlacing theorem; see \cite[Chapter 9]{godroy2001}. Therefore,
    \[
        \frac{1}{\lambda} + \sum_{i = 1}^{n'} \frac{1}{\lambda - \xi_i} - \sum_{i = 1}^{n' - 1} \frac{1}{\lambda - \zeta_i} = \sum_{i = 1}^{n' - 1} \left( \frac{1}{\lambda - \xi_{i + 1}} - \frac{1}{\lambda - \zeta_i} \right) + \frac{1}{\lambda} + \frac{1}{\lambda - \xi_1} > 0 .
    \]
    With all of this in mind, we conclude that $\mathcal{R}_G'(\lambda) > 0$ for any $\lambda > \xi_{n'}$, which means that $\mathcal{R}_G$ is strictly increasing. Since $\mathcal{R}_G$ is continuous and $\mathcal{R}_G(\rho(\mathcal{T}_1(G))) = 0$, it also follows that its image is $[0, +\infty)$.
\end{proof}

From Lemma \ref{rat_lemma}, we directly obtain the next corollary.

\begin{corollary}\label{rat_cor}
    For any $e \ge 1$, $n \ge b_e$ and threshold graph $G \in \mathcal{C}_{n, e}$, let
    \[
        Q(\lambda) = \lambda \, \mathcal{P}_{\mathcal{T}_1(G)}(\lambda) - \alpha \, \mathcal{P}_{\mathcal{T}(G)}(\lambda) ,
    \]
    where $\alpha \ge 0$ is real. Then we have $\rho(Q) = \mathcal{R}_G^{-1}(\alpha)$, alongside $Q(\lambda) < 0$ for any $\lambda \in (\rho(\mathcal{T}(G)), \rho(Q))$.
\end{corollary}
\begin{proof}
    Let $\beta = \mathcal{R}_G^{-1}(\alpha)$, as well as $\xi_1 = \rho(\mathcal{T}_1(G))$ and $\xi = \rho(\mathcal{T}(G))$. By Lemma~\ref{rat_lemma}, we have $Q(\beta) = 0$, alongside $Q(\lambda) > 0$ for any $\lambda > \beta$, and $Q(\lambda) < 0$ for any $\lambda \in [\xi_1, \beta)$. Therefore, $\rho(Q) = \beta$. From the theory of nonnegative matrices, we know that $\xi_1$ is a simple root of $\mathcal{P}_{\mathcal{T}_1(G)}(\lambda)$. By virtue of the eigenvalue interlacing theorem, this means that for any $\lambda \in (\xi, \xi_1)$, we have $\mathcal{P}_{\mathcal{T}_1(G)}(\lambda) < 0$ and $\mathcal{P}_{\mathcal{T}(G)}(\lambda) > 0$, hence $Q(\lambda) < 0$.
\end{proof}

We now define the polynomial
\begin{align*}
    \mathcal{Q}_{G_1, G_2}(\lambda) &= \lambda \left( \mathcal{P}_{\mathcal{T}_1(G_1)}(\lambda) \, \mathcal{P}_{\mathcal{T}(G_2)}(\lambda) - \mathcal{P}_{\mathcal{T}_1(G_2)}(\lambda) \, \mathcal{P}_{\mathcal{T}(G_1)}(\lambda) \right)\\
    &+ (|\mathcal{T}(G_1)| - |\mathcal{T}(G_2)|) \, \mathcal{P}_{\mathcal{T}(G_1)}(\lambda) \, \mathcal{P}_{\mathcal{T}(G_2)}(\lambda) 
\end{align*}
for any $e \ge 1$, $n \ge b_e$ and threshold graphs $G_1, G_2 \in \mathcal{C}_{n, e}$. As it turns out, $\mathcal{Q}_{G_1, G_2}(\lambda)$ can be used while proving the extremal property of $\mathcal{V}_{n, e}$ and $\mathcal{D}_{n, e}$, as shown in the following two lemmas.

\begin{lemma}\label{comp_lemma_1}
    Suppose that $e \ge 4$ and $n \ge e + 2$, and let $G \in \mathcal{C}_{n, e}$ be a threshold graph such that $G \not\cong \mathcal{V}_{n, e}$. Then $\mathcal{Q}_{G, \mathcal{V}_{n, e}}(\lambda)$ is a nonzero polynomial with a positive leading coefficient. Moreover, the following holds:
    \begin{enumerate}[label=\textbf{(\roman*)}]
        \item if $\rho(\mathcal{Q}_{G, \mathcal{V}_{n, e}}) < \rho(\mathcal{T}_1(\mathcal{V}_{n, e}))$, then $\chi > \rho(G)$;
        \item if $\rho(\mathcal{Q}_{G, \mathcal{V}_{n, e}}) \ge \rho(\mathcal{T}_1(\mathcal{V}_{n, e}))$ and $n > e + 2 + \mathcal{R}_{\mathcal{V}_{n, e}}(\rho(\mathcal{Q}_{G, \mathcal{V}_{n, e}}))$, then $\chi > \rho(G)$.
    \end{enumerate}
\end{lemma}
\begin{proof}
    Observe that $\mathcal{T}(\mathcal{V}_{n, e}) = K_{1, e}$ and $\mathcal{T}_1(\mathcal{V}_{n, e}) = K_2 \lor e K_1$. Let $\xi_1 = \rho(K_2 \lor e K_1)$ and note that $\xi_1 > \rho(\mathcal{T}(G))$. Indeed, $\mathcal{T}(G)$ is a graph of size $e$, hence by Theorems \ref{bruhoff_th} and \ref{row_th}, we have
    \[
        \rho(\mathcal{T}(G)) \le \rho(K_t \lor (K_{k - t} + K_1)) < \rho(K_{k + 1}) = k.
    \]
    On the other hand, Lemma \ref{third_lemma} implies
    \[
        \xi_1 = \frac{1 + \sqrt{8e + 1}}{2} \ge \frac{1 + \sqrt{4k(k - 1) + 1}}{2} = \frac{1 + (2k - 1)}{2} = k.
    \]
    Note that $\chi \ge \xi_1$ due to Lemmas \ref{poly_2_lemma} and \ref{rat_lemma}. Therefore, by Lemma \ref{poly_2_lemma} and Corollary~\ref{rat_cor}, we conclude that $\chi > \rho(G)$ holds if and only if
    \[
        \chi \, \mathcal{P}_{\mathcal{T}_1(G)}(\chi) - (n - n') \, \mathcal{P}_{\mathcal{T}(G)}(\chi) > 0,
    \]
    where $n' = |\mathcal{T}_1(G)| \le e + 1$. Since
    \[
        n - e - 2 = \frac{\chi \, \mathcal{P}_{K_2 \lor eK_1}(\chi)}{\mathcal{P}_{K_{1, e}}(\chi)},
    \]
    a routine computation yields that $\chi > \rho(G)$ is equivalent to $\mathcal{Q}_{G, \mathcal{V}_{n, e}}(\chi) > 0$.

    If we suppose that $\mathcal{Q}_{G, \mathcal{V}_{n, e}}(\lambda)$ is a zero polynomial or has a negative leading coefficient, then this implies that there is a $\zeta \ge \xi_1$, such that $\chi \le \rho(G)$ provided $\chi \ge \zeta$. Bearing in mind Lemma \ref{rat_lemma}, we obtain a contradiction to Theorem \ref{cvetrow_v_th}. Therefore, $\mathcal{Q}_{G, \mathcal{V}_{n, e}}(\lambda)$ must be a nonzero polynomial with a positive leading coefficient. If $\rho(\mathcal{Q}_{G, \mathcal{V}_{n, e}}) < \xi_1$, then $\chi \ge \xi_1$ implies $\mathcal{Q}_{G, \mathcal{V}_{n, e}}(\chi) > 0$, hence $\chi > \rho(G)$. This proves statement (i). Now, suppose that $\rho(\mathcal{Q}_{G, \mathcal{V}_{n, e}}) \ge \xi_1$. In this case, the value $\mathcal{R}_{\mathcal{V}_{n, e}}(\rho(\mathcal{Q}_{G, \mathcal{V}_{n, e}}))$ is well-defined. Furthermore, if $n - e - 2 > \mathcal{R}_{\mathcal{V}_{n, e}}(\rho(\mathcal{Q}_{G, \mathcal{V}_{n, e}}))$, then Lemmas \ref{poly_2_lemma} and \ref{rat_lemma} imply $\chi > \rho(\mathcal{Q}_{G, \mathcal{V}_{n, e}})$. From here, we get $\chi > \rho(G)$, which proves statement (ii).
\end{proof}

\begin{lemma}\label{comp_lemma_2}
    Suppose that $e \ge 4$, $t_e \ge 1$ and $n \ge b_e$, and let $G \in \mathcal{C}_{n, e}$ be a threshold graph such that $G \not\cong \mathcal{D}_{n, e}$. Then the following holds:
    \begin{enumerate}[label=\textbf{(\roman*)}]
        \item if $\mathcal{Q}_{G, \mathcal{D}_{n, e}}(\lambda)$ is a nonzero polynomial with a positive leading coefficient such that \linebreak $\rho(\mathcal{Q}_{G, \mathcal{D}_{n, e}}) < \rho(\mathcal{T}_1(\mathcal{D}_{n, e}))$, then $\gamma > \rho(G)$;

        \item if $\mathcal{Q}_{G, \mathcal{D}_{n, e}}(\lambda)$ is a nonzero polynomial with a negative leading coefficient such that \linebreak $\rho_2(\mathcal{Q}_{G, \mathcal{D}_{n, e}}) < \rho(\mathcal{T}_1(\mathcal{D}_{n, e})) < \rho(\mathcal{Q}_{G, \mathcal{D}_{n, e}})$ and $n < b_e + \mathcal{R}_{\mathcal{D}_{n, e}}(\rho(\mathcal{Q}_{G, \mathcal{D}_{n, e}}))$, then $\gamma > \rho(G)$.
    \end{enumerate}
\end{lemma}
\begin{proof}
    Observe that $\mathcal{T}(\mathcal{D}_{n, e}) = K_t \lor \left(K_{k - t} + K_1\right)$ and $\mathcal{T}_1(\mathcal{D}_{n, e}) = K_{t + 1} \lor \left(K_{k - t} + K_1\right)$. Recall that $b_e = k_e + 2$ and let $\xi_1 = \rho(K_{t + 1} \lor \left(K_{k - t} + K_1\right))$. We trivially verify that $\xi_1 > k$, hence we can conclude analogously to Lemma~\ref{comp_lemma_1} that $\xi_1 > \rho(\mathcal{T}(G))$, which implies that $\gamma > \rho(G)$ is equivalent to $\mathcal{Q}_{G, \mathcal{D}_{n, e}}(\gamma) > 0$. If $\mathcal{Q}_{G, \mathcal{D}_{n, e}}(\lambda)$ is a nonzero polynomial with a positive leading coefficient such that $\rho(\mathcal{Q}_{G, \mathcal{D}_{n, e}}) < \xi_1$, then $\gamma \ge \xi_1$ gives $\mathcal{Q}_{G, \mathcal{D}_{n, e}}(\gamma) > 0$. Therefore, we obtain $\gamma > \rho(G)$ in accordance with statement (i).

    Now, suppose that $\mathcal{Q}_{G, \mathcal{D}_{n, e}}(\lambda)$ is a nonzero polynomial with a negative leading coefficient such that $\rho_2(\mathcal{Q}_{G, \mathcal{D}_{n, e}}) < \xi_1 < \rho(\mathcal{Q}_{G, \mathcal{D}_{n, e}})$ and $n < b_e + \mathcal{R}_{\mathcal{D}_{n, e}}(\rho(\mathcal{Q}_{G, \mathcal{D}_{n, e}}))$. In this case, $\rho(\mathcal{Q}_{G, \mathcal{D}_{n, e}}) \neq -\infty$ must be a simple root of $\mathcal{Q}_{G, \mathcal{D}_{n, e}}(\lambda)$ due to $\rho_2(\mathcal{Q}_{G, \mathcal{D}_{n, e}}) < \rho(\mathcal{Q}_{G, \mathcal{D}_{n, e}})$. From here, we get $\mathcal{Q}_{G, \mathcal{D}_{n, e}}(\lambda) > 0$ for any $\lambda \in (\rho_2(\mathcal{Q}_{G, \mathcal{D}_{n, e}}), \rho(\mathcal{Q}_{G, \mathcal{D}_{n, e}}))$. By Lemma~\ref{rat_lemma}, we have that $\mathcal{R}_{\mathcal{D}_{n, e}}(\rho(\mathcal{Q}_{G, \mathcal{D}_{n, e}}))$ is well-defined. Moreover, $n - b_e < \mathcal{R}_{\mathcal{D}_{n, e}}(\rho(\mathcal{Q}_{G, \mathcal{D}_{n, e}}))$ implies $\gamma < \rho(\mathcal{Q}_{G, \mathcal{D}_{n, e}})$, while $\gamma > \rho_2(\mathcal{Q}_{G, \mathcal{D}_{n, e}})$ also holds because $\gamma \ge \xi_1$. With all of this in mind, we conclude that $\mathcal{Q}_{G, \mathcal{D}_{n, e}}(\gamma) > 0$, hence $\gamma > \rho(G)$, which yields statement~(ii).
\end{proof}

We are now in a position to describe a strategy that can be used to construct a computer-assisted proof of Proposition \ref{computer_prop}. We present an algorithm that inputs an argument $e \ge 4$ such that $t_e \ge 1$ and attempts to apply Lemmas \ref{comp_lemma_1} and \ref{comp_lemma_2} to verify the claim that no threshold graph $G \in \mathcal{C}_{n, e}$ such that $G \not\cong \mathcal{D}_{n, e}, \mathcal{V}_{n, e}$ can have the maximum spectral radius on $\mathcal{C}_{n, e}$, for any $n \ge b_e$. For a given $e \ge 4$ with $t_e \ge 1$, let $\mathcal{S}_e$ be the finite set of all the graphs that appear as the T-subgraph of a threshold graph from $\mathcal{C}_{n, e}$ for some $n \ge b_e$. Moreover, let $\mathcal{S}^*_e = \mathcal{S}_e \setminus \{ K_{1, e}, K_t \lor \left(K_{k - t} + K_1\right) \}$. The algorithm iterates over all the graphs $T \in \mathcal{S}^*_e$ and performs the following procedure:

\begin{enumerate}[label=\textbf{(\arabic*)}]
    \item Let $G \in \mathcal{C}_{n, e}$ be a threshold graph such that $\mathcal{T}(G) = T$. Note that $n \ge |T \lor K_1|$ and $b_e \le |T \lor K_1| \le e + 1$.
    
    \item Verify that $\mathcal{Q}_{G, \mathcal{D}_{n, e}}(\lambda) \not\equiv 0$.
    
    \item If the leading coefficient of $\mathcal{Q}_{G, \mathcal{D}_{n, e}}(\lambda)$ is positive, verify that $\rho(\mathcal{Q}_{G, \mathcal{D}_{n, e}}(\lambda)) < \rho(\mathcal{T}_1(\mathcal{D}_{n, e}))$ and use Lemma \ref{comp_lemma_2}(i) to conclude that $G$ does not attain the maximum spectral radius on $\mathcal{C}_{n, e}$. Report success and stop the algorithm.

    \item If the leading coefficient of $\mathcal{Q}_{G, \mathcal{D}_{n, e}}(\lambda)$ is negative, verify that $\rho_2(\mathcal{Q}_{G, \mathcal{D}_{n, e}}) < \linebreak \rho(\mathcal{T}_1(\mathcal{D}_{n, e})) < \rho(\mathcal{Q}_{G, \mathcal{D}_{n, e}})$ and use Lemma \ref{comp_lemma_2}(ii) to conclude that $G$ does not attain the maximum spectral radius on $\mathcal{C}_{n, e}$, provided $n < n_U := b + \mathcal{R}_{\mathcal{D}_{n, e}}(\rho(\mathcal{Q}_{G, \mathcal{D}_{n, e}}))$.

    \item Note that $\rho(\mathcal{Q}_{G, \mathcal{V}_{n, e}}) \not\equiv 0$, due to Lemma \ref{comp_lemma_1}. If $\rho(\mathcal{Q}_{G, \mathcal{V}_{n, e}}) < \rho(\mathcal{T}_1(\mathcal{V}_{n, e}))$, use Lemma \ref{comp_lemma_1}(i) to conclude that $G$ does not attain the maximum spectral radius on $\mathcal{C}_{n, e}$, provided $n \ge n_L := e + 2$.

    \item On the other hand, if $\rho(\mathcal{Q}_{G, \mathcal{V}_{n, e}}) \ge \rho(\mathcal{T}_1(\mathcal{V}_{n, e}))$, use Lemma \ref{comp_lemma_1}(ii) to conclude that $G$ does not attain the maximum spectral radius on $\mathcal{C}_{n, e}$, provided $n > n_L := e + 2 + \mathcal{R}_{\mathcal{V}_{n, e}}(\rho(\mathcal{Q}_{G, \mathcal{V}_{n, e}}))$.

    \item Verify that $n_U > n_L$, thus showing that at least one of the graphs $\mathcal{D}_{n, e}$ and $\mathcal{V}_{n, e}$ has a larger spectral radius than $G$, regardless of what the value of $n$ is. Report success.
\end{enumerate}

The algorithm reports failure and stops if any of the verifications does not pass for any $T \in \mathcal{S}^*_e$. It can be conveniently implemented using \texttt{SageMath} \cite{SageMath} as shown in \cite{GitHub}. The developed program uses a strictly decreasing sequence of positive integers $(s_1, s_2, \ldots, s_c)$ to represent a T-subgraph $T$ so that $c$ is maximal such that $A_{c, c + 1} = 1$, and for each $i \in \{1, 2, \ldots, c \}$, the value $s_i$ is maximal such that $A_{i, i + s_i} = 1$, with $A$ being the (stepwise) adjacency matrix of $T$. With this in mind, generating all the members of $\mathcal{S}^*_e$ gets down to finding all the ways in which $e$ can be represented as a decreasing sum of positive integers. This can be efficiently done, e.g., via the auxiliary function \texttt{generateIncreasing} implemented in \texttt{C++20} in \cite{StoDam2025}. The developed \texttt{SageMath} script can then be applied to iterate over all of these T-subgraphs and perform the verification algorithm on each of them. Since the algorithm reports success for any $e \in \{4, 5, \ldots, 130 \}$ such that $t_e \ge 1$, Proposition~\ref{computer_prop} follows from the obtained computational results and Theorems \ref{con_thresh_th} and \ref{bell_th}. We end the section by noting that the algorithm implementation can be optimized by using the following two claims which can easily be verified.

\begin{lemma}
    For any $e \ge 1$ such that $t_e \ge 1$, and $n_e \ge b_e$, we have
    \[
        \mathcal{R}_{\mathcal{D}_{n, e}}(\lambda) = \frac{\lambda(\lambda + 1)(\lambda^3 - \lambda^2(k - 1) - \lambda(k + t + 1) + (t + 1)(k - t - 1))}{\lambda^3 - \lambda^2(k - 2) - \lambda(k + t - 1) + t(k - t - 1)}
    \]
    and $\rho(\lambda^3 - \lambda^2(k - 1) - \lambda(k + t + 1) + (t + 1)(k - t - 1)) = \rho(\mathcal{T}_1(\mathcal{D}_{n, e}))$.
\end{lemma}

\begin{lemma}\label{v_simple_lemma}
    For any $e \ge 1$ and $n \ge e + 2$, we have
    \[
        \mathcal{R}_{\mathcal{V}_{n, e}}(\lambda) = \frac{\lambda(\lambda + 1)(\lambda^2 - \lambda - 2e)}{\lambda^2 - e}
    \]
    and $\rho(\lambda^2 - \lambda - 2e) = \rho(\mathcal{T}_1(\mathcal{V}_{n, e}))$.
\end{lemma}

\section{Comparison of graphs \texorpdfstring{$\mathcal{D}_{n, e}$}{D(n, e)} and \texorpdfstring{$\mathcal{V}_{n, e}$}{V(n, e)}}

In this section, we provide a criterion that compares $\gamma$ and $\chi$. We begin by introducing the polynomial
\begin{align*}
    \Psi_e(\lambda) &= \lambda^3 (k - 1)(k - 2)(k^2 - 3k + 4t)\\
    &- \lambda^2 (k^5 - 6k^4 + k^3 (4t + 15) - k^2(20t + 18) + k(8t^2 + 24t + 8) - (4t^2 + 12t))\\
    &- \lambda (k^2 - k + 2t)(k^3 + k^2(t - 4) - k(3t - 5) + (4t^2 - 2t - 2))\\
    &+ t(k - t - 1)(k^2 - 3k + 2t)(k^2 - k + 2t) .
\end{align*}
for every $e \ge 4$. Now, let $\psi_e$ be the largest real root of $\Psi_e(\lambda)$ and let
\[
    \omega_e = e + 2 + \frac{\psi_e(\psi_e + 1)(\psi_e^2 - \psi_e - 2e)}{\psi_e^2 - e}.
\]
With this in mind, we can state the two main results of the present section as follows.

\begin{proposition}\label{comparison_prop}
For any $e \ge 4$ and $n \ge e + 2$, exactly one of the following three statements holds:
\begin{enumerate}[label=\textbf{(\roman*)}]
    \item $\rho(\mathcal{V}_{n, e}) < \rho(\mathcal{D}_{n, e}) < \psi_e$;
    \item $\rho(\mathcal{V}_{n, e}) = \rho(\mathcal{D}_{n, e}) = \psi_e$;
    \item $\rho(\mathcal{V}_{n, e}) > \rho(\mathcal{D}_{n, e}) > \psi_e$.
\end{enumerate}
\end{proposition}

\begin{proposition}\label{activation_prop}
For any $e \ge 4$, we have $\omega_e > e + 2$, alongside the following:
\begin{enumerate}[label=\textbf{(\roman*)}]
    \item if $b_e \le n < \omega_e$, then $\rho(\mathcal{D}_{n, e}) < \psi_e$;
    \item if $n = \omega_e$, then $\rho(\mathcal{D}_{n, e}) = \psi_e$;
    \item if $n > \omega_e$, then $\rho(\mathcal{D}_{n, e}) > \psi_e$.
\end{enumerate}
\end{proposition}

We start by showing that $\psi_e$ and $\omega_e$ are well-defined for any $e \ge 4$ using the next two lemmas.

\begin{lemma}\label{psi_cubic_lemma}
    For any $e \ge 4$, the polynomial $\Psi_e(\lambda)$ is cubic and its leading coefficient is positive.
\end{lemma}
\begin{proof}
    Let $P(e) = (k - 1)(k - 2)(k^2 - 3k + 4t)$. We trivially observe that $P(4) = 8$ and $P(5) = 16$. Now, suppose that $e \ge 6$. In this case, we have $k \ge 4$ and $t \ge 0$, hence
    \[
        \pushQED{\qed}
        P(e) \ge (k - 1)(k - 2)(k^2 - 3k) = k(k - 1)(k - 2)(k - 3) > 0. \qedhere
    \]
\end{proof}

\begin{lemma}\label{psi_value_lemma}
    For any $e \ge 4$, we have $\psi_e > k_e + 1$.
\end{lemma}
\begin{proof}
    If $t = 0$, then $k \ge 4$ and we have
    \[ 
        \Psi_e(\lambda) = k(k - 1)(k - 2) \, \lambda(\lambda + 1)(\lambda(k - 3) - (k - 1)^2),
    \]
    hence
    \[
        \psi_e = \frac{(k - 1)^2}{k - 3} = k + 1 + \frac{4}{k - 3} > k + 1.
    \]
    On the other hand, if $t = 1, 2, 3$, then we get
    \begin{align*}
        \Psi_e(k + 1) &= -(k - 1)(k - 2)(4k^3 + 5k^2 - 19k - 8),\\
        \Psi_e(k + 1) &= -4(k^5 - 3k^4 - 5k^3 + 47k^2 - 28k + 20),\\
        \Psi_e(k + 1) &= -4k^5 + 15k^4 - 2k^3 - 411k^2 + 302k - 504,
    \end{align*}
    respectively. It is not difficult to verify that $\Psi_e(k + 1) < 0$ holds in each of these cases, which means that $\psi_e > k + 1$, by Lemma \ref{psi_cubic_lemma}.
    
    Now, suppose that $t \ge 4$. In this case, a routine computation leads to
    \begin{align*}
        \Psi_e(k + 1) &= -4t^4 - t^3 (4k^2 - 4k + 12) - t^2 (k^4 + 6k^3 + 23k^2 - 22k - 8)\\
        &+ t(8k^4 - 12k^3 - 32k^2 + 4k + 24) - (4k^5 - 20k^3 + 16k) .
    \end{align*}
    Since $4k^2 - 4k + 12, k^4 + 6k^3 + 23k^2 - 22k - 8 > 0$ alongside $t^2 \ge 4t$, $t^3 \ge 16t$ and $t^4 \ge 64t$, we have
    \begin{align*}
        \Psi_e(k + 1) &= -256t - 16t (4k^2 - 4k + 12) - 4t (k^4 + 6k^3 + 23k^2 - 22k - 8)\\
        &+ t(8k^4 - 12k^3 - 32k^2 + 4k + 24) - (4k^5 - 20k^3 + 16k) ,
    \end{align*}
    which implies
    \begin{equation}\label{comp_aux_4}
        \Psi_e(k + 1) \le t(4k^4 - 12k^3 - 188k^2 + 156k - 392) - (4k^5 - 20k^3 + 16k) .
    \end{equation}

    If $4k^4 - 12k^3 - 188k^2 + 156k - 392 \le 0$, then we obtain
    \[
        \Psi_e(k + 1) = -(4k^5 - 20k^3 + 16k) < 0,
    \]
    which yields $\psi_e > k + 1$. On the other hand, if $4k^4 - 12k^3 - 188k^2 + 156k - 392 > 0$, then \eqref{comp_aux_4} implies
    \begin{align*}
        \Psi_e(k + 1) &\le (k - 1)(4k^4 - 12k^3 - 188k^2 + 156k - 392) - (4k^5 - 20k^3 + 16k)\\
        &= -4(k - 1)(4k^3 + 43k^2 - 43k + 98) < 0.
    \end{align*}
    In this case, we also get $\psi_e > k + 1$.
\end{proof}

Lemma \ref{psi_cubic_lemma} verifies that $\Psi_e(\lambda)$ has a real root, hence $\psi_e$ is well-defined. On the other hand, Lemma \ref{psi_value_lemma} shows that $\psi_e^2 - \psi_e - 2e > 0$, since
\[
    \frac{1 + \sqrt{8e + 1}}{2} < \frac{1 + \sqrt{4k(k + 1) + 1}}{2} = \frac{1 + (2k + 1)}{2} = k + 1 < \psi_e.
\]
Therefore, for each $e \ge 4$, the value $\omega_e$ is well-defined and we have $\omega_e > e + 2$. We proceed with the following lemma on the spectral properties of $\mathcal{D}_{n, e}$.

\begin{lemma}\label{gnm_technical_lemma}
    Suppose that $e \ge 4$, $t_e \ge 1$ and $n \ge b_e$, and let $Y_i = \dfrac{y_i}{y_1}$ for $i \in \{ 2, k_e + 1, \linebreak k_e + 2 \}$. Then we have
    \begin{align*}
        Y_2 &= \frac{\gamma(\gamma + 2) - k + t + 1}{\gamma(\gamma + 1)(\gamma - k + t + 1) - t(\gamma(\gamma + 2) - k + t + 1)},\\
        Y_{k + 1} &= \frac{\gamma(\gamma + 1)}{\gamma(\gamma + 1)(\gamma - k + t + 1) - t(\gamma(\gamma + 2) - k + t + 1)},\\
        Y_{k + 2} &= \frac{(\gamma + 1)(\gamma - k + t + 1)}{\gamma(\gamma + 1)(\gamma - k + t + 1) - t(\gamma(\gamma + 2) - k + t + 1)}.
    \end{align*}
\end{lemma}
\begin{proof}
    The equations \eqref{gnm_local_2}--\eqref{gnm_local_4} transform into
    \begin{align}
        \label{exaux_3} (\gamma + 1) Y_2 &= 1 + t Y_2 + (k - t) Y_{k + 1} + Y_{k + 2},\\
        \label{exaux_4} (\gamma + 1) Y_{k + 1} &= 1 + t Y_2 + (k - t) Y_{k + 1},\\
        \label{exaux_5} \gamma Y_{k + 2} &= 1 + t Y_2,
    \end{align}
    respectively. By combining \eqref{exaux_4} and \eqref{exaux_5}, we get
    \[
        (\gamma + 1) Y_{k + 1} = \gamma Y_{k + 2} + (k - t) Y_{k + 1},
    \]
    hence
    \begin{equation}\label{exaux_6}
        \frac{Y_{k + 1}}{Y_{k + 2}} = \frac{\gamma}{\gamma - k + t + 1} .
    \end{equation}
    Also, by multiplying \eqref{exaux_3} with $\gamma$ and plugging in \eqref{exaux_4} and \eqref{exaux_6}, we obtain
    \[
        \gamma(\gamma + 1) Y_2 = \gamma (\gamma + 1) Y_{k + 1} + (\gamma - k + t + 1) Y_{k + 1},
    \]
    which implies
    \begin{equation}\label{exaux_7}
        \frac{Y_2}{Y_{k + 1}} = \frac{\gamma(\gamma + 2) - k + t + 1}{\gamma(\gamma + 1)} .
    \end{equation}
    From \eqref{exaux_6} and \eqref{exaux_7}, we conclude that
    \begin{align*}
        Y_2 &= \xi \left( \gamma(\gamma + 2) - k + t + 1 \right),\\
        Y_{k + 1} &= \xi \gamma (\gamma + 1),\\
        Y_{k + 2} &= \xi (\gamma + 1)(\gamma - k + t + 1)
    \end{align*}
    holds for some $\xi > 0$. By using \eqref{exaux_5}, a routine computation now leads to
    \[
        \pushQED{\qed}
        \xi = \frac{1}{\gamma(\gamma + 1)(\gamma - k + t + 1) - t(\gamma(\gamma + 2) - k + t + 1)} . \qedhere
    \]
\end{proof}

We also need the next technical lemma on the monotonicity of $\lambda \mapsto \Psi_e(\lambda)$.

\begin{lemma}\label{increasing_lemma}
    For any $e \ge 28$, the function $\lambda \mapsto \Psi_e(\lambda)$ is strictly increasing on $[k_e, +\infty)$.
\end{lemma}
\begin{proof}
    Observe that
    \begin{align*}
        \Psi'_e(\lambda) &= 3\lambda^2 (k - 1)(k - 2)(k^2 - 3k + 4t)\\
        &- 2\lambda (k^5 - 6k^4 + k^3 (4t + 15) - k^2(20t + 18) + k(8t^2 + 24t + 8) - (4t^2 + 12t))\\
        &- (k^2 - k + 2t)(k^3 + k^2(t - 4) - k(3t - 5) + (4t^2 - 2t - 2))
    \end{align*}
    and
    \begin{align*}
        \Psi''_e(\lambda) &= 6\lambda (k - 1)(k - 2)(k^2 - 3k + 4t)\\
        &- 2(k^5 - 6k^4 + k^3 (4t + 15) - k^2(20t + 18) + k(8t^2 + 24t + 8) - (4t^2 + 12t)) .
    \end{align*}
    From Lemma~\ref{psi_cubic_lemma}, we conclude that $\lambda \mapsto \Psi''_e(\lambda)$ is strictly increasing, hence for any $\lambda \ge k$, we have
    \begin{align*}
        \Psi''_e(\lambda) &\ge 6k(k - 1)(k - 2)(k^2 - 3k + 4t)\\
        &- 2(k^5 - 6k^4 + k^3 (4t + 15) - k^2(20t + 18) + k(8t^2 + 24t + 8) - (4t^2 + 12t)) ,
    \end{align*}
    i.e.,
    \[
        \Psi''_e(\lambda) \ge -t^2(16k - 8) + t(16k^3 - 32k^2 + 24) + (4k^5 - 24k^4 + 36k^3 - 16k) .
    \]
    Since $0 \le t \le k - 1$, we further obtain
    \begin{align*}
        \Psi''_e(\lambda) &\ge -(k - 1)^2(16k - 8) + (4k^5 - 24k^4 + 36k^3 - 16k)\\
        &= 4(k - 1)(k^4 - 5k^3 + 10k - 2)
    \end{align*}
    for any $\lambda \ge k$. Since $k \ge 5$, we trivially observe that $k^4 - 5k^3 + 10k - 2 > 0$, hence $\lambda \mapsto \Psi_e'(\lambda)$ is strictly increasing on $[k, +\infty)$.
    
    A routine computation gives
    \begin{align*}
        \Psi'_e(k) &= -8t^3 - t^2(22k^2 - 18k - 4) + t(3k^4 + 6k^3 - 17k^2 + 12k + 4)\\
        &+ (k^6 - 7k^5 + 8k^4 + 9k^3 - 9k^2 - 2k) .
    \end{align*}
    Due to $0 \le t \le k - 1$, we have
    \begin{align*}
        \Psi'_e(k) &\ge (k^6 - 7k^5 + 8k^4 + 9k^3 - 9k^2 - 2k) - (k - 1)^2 (22k^2 - 18k - 4) - 8(k - 1)^3\\
        &= (k - 1)(k^5 - 6k^4 - 20k^3 + 43k^2 + 4k - 12) .
    \end{align*}
    It is not difficult to verify that $k^5 - 6k^4 - 20k^3 + 43k^2 + 4k - 12 > 0$, provided $k \ge 8$. Therefore, we have $\Psi'_e(\lambda) > 0$ for any $\lambda \ge k$, which means that $\lambda \mapsto \Psi_e(\lambda)$ is strictly increasing on $[k, +\infty)$.
\end{proof}

We are now in a position to complete the proof of Proposition \ref{comparison_prop} as follows.

\begin{proof}[Proof of Proposition \ref{comparison_prop}]
    First, assume that $t = 0$. As noted in the proof of Lemma \ref{psi_value_lemma}, in this case we have $\psi_e = k + 1 + \frac{4}{k - 3}$. Therefore, the result follows from \cite[Lemma~1]{bell1991} for the case $n \ge e + 3$, while it is trivial to extend this proof to also cover the case $n = e + 2$. 
    
    Now, suppose that $t \ge 1$. We follow Rowlinson \cite{rowlinson1988} in observing that $y^\intercal (B - C) z = (\gamma - \chi) y^\intercal z$, where $y^\intercal z > 0$. 
    Since $t \ge 1$, the matrix $B - C$ has $e - k \ge 1$ entries above the main diagonal that are equal to one. Therefore, $y^\intercal (B - C) z = \alpha - \beta$, where
    \[
        \beta = y_2 (z_{k + 3} + z_{k + 4} + \cdots + z_{e + 2}) + z_2 (y_{k + 3} + y_{k + 4} + \cdots + y_{e + 2}) = (e - k)(y_2 z_3 + z_2 y_n),
    \]
    while
    \begin{align*}
        \alpha &= \sum_{i = 3}^{t + 1} \sum_{j = i + 1}^{k + 2} (y_i z_j + z_i y_j) + \sum_{i = t + 2}^{k} \sum_{j = i + 1}^{k + 1} (y_i z_j + z_i y_j)\\
        &= z_3 \left( \sum_{i = 3}^{t + 1} \sum_{j = i + 1}^{k + 2} (y_i + y_j) + \sum_{i = t + 2}^{k} \sum_{j = i + 1}^{k + 1} (y_i + y_j) \right)\\
        &= z_3 \left( \sum_{i = 3}^{t + 1} \sum_{j = i + 1}^{k + 2} (y_2 + y_j) + \sum_{i = t + 2}^{k} \sum_{j = i + 1}^{k + 1} 2y_{k + 1} \right)\\
        &= z_3 \left( \sum_{i = 3}^{t + 1} ((k + t - 2i + 3)y_2 + (k - t)y_{k + 1} + y_{k + 2}) + \sum_{i = t + 2}^{k} 2(k - i + 1)y_{k + 1} \right)\\
        &= z_3 \left( (t - 1)y_{k + 2} + (k - 1)(t - 1) y_2 + (k - t)(k - 2) y_{k + 1} \right) .
    \end{align*}

    Suppose that $\gamma > \chi$. In this case, we have $\alpha > \beta$, hence
    \[
        (e - k) z_2 y_n < z_3 \left( \left( (k - 1)(t - 1) - (e - k)\right)y_2 + (k - 2)(k - t)y_{k + 1} + (t - 1)y_{k + 2} \right),
    \]
    i.e.,
    \begin{equation}\label{exaux_1}
        \frac{z_2}{z_3} < \frac{\left( (k - 1)(t - 1) - (e - k)\right)y_2 + (k - 2)(k - t)y_{k + 1} + (t - 1)y_{k + 2}}{(e - k)y_n} .
    \end{equation}
    By combining \eqref{hnm_local_2} and \eqref{hnm_local_3}, we get $(\chi + 1)z_2 = \chi z_3 + e z_3$, which leads us to
    \begin{equation}\label{exaux_2}
        \frac{z_2}{z_3} = \frac{\chi + e}{\chi + 1} = 1 + \frac{e - 1}{\chi + 1} > 1 + \frac{e - 1}{\gamma + 1} = \frac{\gamma + e}{\gamma + 1}.
    \end{equation}
    Let $Y_i = \dfrac{y_i}{y_1}$ for $i \in \{2, k + 1, k + 2, n \}$ and note that $Y_n = \frac{1}{\gamma}$. Thus, from \eqref{exaux_1} and \eqref{exaux_2}, we obtain
    \[
        \frac{(e - k)(\gamma + e)}{\gamma(\gamma + 1)} < \left( (k - 1)(t - 1) - (e - k)\right)Y_2 + (k - 2)(k - t)Y_{k + 1} + (t - 1)Y_{k + 2} .
    \]
    By applying Lemma \ref{gnm_technical_lemma} and performing a routine computation, we reach $\Psi_e(\gamma) < 0$. Similarly, we can show that $\Psi_e(\gamma) > 0$ if $\gamma < \chi$, and $\Psi_e(\gamma) = 0$ if $\gamma = \chi$.

    From Lemma \ref{psi_value_lemma}, we know that $\psi_e > k$. If $4 \le e \le 27$, then with the help of any adequate mathematical software, we can verify that $\Psi_e(\lambda)$ has three distinct real roots such that $\psi_e$ is the only root from $[k, +\infty)$. On the other hand, if $e \ge 28$, then Lemma~\ref{increasing_lemma} implies that $\lambda \mapsto \Psi_e(\lambda)$ is strictly increasing on $[k, +\infty)$, which means that $\psi_e$ is the only root of $\Psi_e(\lambda)$ from $[k, +\infty)$. Therefore, in any case, we have $\Psi_e(\lambda) > 0$ for any $\lambda \in (\psi_e, +\infty)$, and $\Psi_e(\lambda) < 0$ for any $\lambda \in [k, \psi_e)$.

    Finally, note that $\mathcal{D}_{n, e}$ contains $K_{k + 1}$ as a proper subgraph, hence $\gamma > k$. With this in mind, we conclude that $\Psi_e(\gamma) > 0$ if and only if $\gamma > \psi_e$, $\Psi_e(\gamma) = 0$ if and only if $\gamma = \psi_e$, and $\Psi_e(\gamma) < 0$ if and only if $\gamma < \psi_e$. Thus, we have that exactly one of the three statements (i), (ii) and (iii) from Proposition \ref{comparison_prop} holds.
\end{proof}

Proposition \ref{activation_prop} can now be proved by using Proposition \ref{comparison_prop} together with the results from Section \ref{sc_tsub}.

\begin{proof}[Proof of Proposition \ref{activation_prop}]
     We have shown that $\omega_e > e + 2$. Suppose that $n \ge e + 2$. As already noted, we have
    \[
        \psi_e > \frac{1 + \sqrt{8e + 1}}{2} = \rho(K_2 \lor eK_1) = \rho(\mathcal{T}_1(\mathcal{V}_{n, e})).
    \]
    With this in mind, Lemma \ref{v_simple_lemma} yields $\omega_e = e + 2 + \mathcal{R}_{\mathcal{V}_{n, e}}(\psi_e)$. Since $|\mathcal{T}_1(\mathcal{V}_{n, e})| = e + 2$, by Lemma \ref{poly_2_lemma}, Lemma~\ref{rat_lemma} and Corollary~\ref{rat_cor}, we conclude that $\chi = \psi_e$ holds if $n = \omega_e$, while $\chi > \psi_e$ holds if $n > \omega_e$, and $\chi < \psi_e$ holds if $e + 2 \le n < \omega_e$. Proposition \ref{comparison_prop} implies that $\gamma = \psi_e$ holds if $n = \omega_e$, while $\gamma > \psi_e$ holds if $n > \omega_n$, and $\gamma < \psi_e$ holds if $e + 2 \le n < \omega_e$. This proves statements (ii) and (iii), and partially proves statement~(i) for the case $e + 2 \le n < \omega_n$. Now, suppose that $b \le n < e + 2$. In this case, $\gamma < \psi_e$ follows from Lemma \ref{poly_2_lemma}, Lemma \ref{rat_lemma}, Corollary \ref{rat_cor} and the fact that $\gamma < \psi_e$ is satisfied when $n = e + 2$, due to $\omega_e > e + 2$.
\end{proof}

By virtue of Propositions~\ref{computer_prop}, \ref{comparison_prop} and \ref{activation_prop}, we obtain the solution to the spectral radius maximization problem on $\mathcal{C}_{n, e}$ for the case $4 \le e \le 130$.

\begin{theorem}
    For any $e \in \{4, 5, \ldots, 130 \}$ and $n \ge b_e$, we have:
    \begin{enumerate}[label=\textbf{(\roman*)}]
        \item if $b_e \le n < \omega_e$, then $\mathcal{D}_{n, e}$ is the unique graph attaining the maximum spectral radius on $\mathcal{C}_{n, e}$;
        \item if $n = \omega_e$, then $\mathcal{D}_{n, e}$ and $\mathcal{V}_{n, e}$ are the only two graphs attaining the maximum spectral radius on $\mathcal{C}_{n, e}$;
        \item if $n > \omega_e$, then $\mathcal{V}_{n, e}$ is the unique graph attaining the maximum spectral radius on $\mathcal{C}_{n, e}$.
    \end{enumerate}
\end{theorem}

\section{Extremality of graph \texorpdfstring{$\mathcal{V}_{n, e}$}{V(n, e}}

In the last section, we apply a strategy inspired by Bell to investigate the spectral radius maximization problem on $\mathcal{C}_{n, e}$ for the remaining case when $e > 130$. In this case, we prove that $\mathcal{V}_{n, e}$ is the unique solution whenever $n \ge e + 2 + 13 \sqrt{e}$, thereby extending Theorem~\ref{cvetrow_v_th}, where the same extremal result was proved with a lower bound $f(e)$ on $n$ that is larger than $e^2 (e + 2)^2$ for $e \ge 7$. Our main result is given in the following theorem.

\begin{theorem}\label{last_th}
    Let $e \ge 5$ be such that $t_e \ge 1$, and let $\ell_e = \dfrac{e k_e}{e - k_e - 1}$. Then for any
    \begin{equation}\label{last_form}
        n > e + 2 + \frac{\ell_e(\ell_e + 1)(\ell_e^2 - \ell_e - 2e)}{\ell_e^2 - e} ,
    \end{equation}
    the graph $\mathcal{V}_{n, e}$ is the unique graph that attains the maximum spectral radius on $\mathcal{C}_{n, e}$.
\end{theorem}
\begin{proof}
    First of all, note that $e > k + 1$ for each $e \ge 5$, hence $\ell_e$ is well-defined. It is straightforward to verify that $\ell_e > k + 1$. Indeed, $\ell_e > k + 1$ is equivalent to $ek > (k + 1)(e - k - 1)$, i.e., $e < (k + 1)^2$, which trivially holds. Since $\rho(K_2 \lor eK_1) < k + 1$, we have that the right-hand side of \eqref{last_form} is also well-defined and it is larger than $e + 2$.

    Let $G \in \mathcal{C}_{n, e}$ be a graph that attains the maximum spectral radius on $\mathcal{C}_{n, e}$. By way of contradiction, suppose that $G \not\cong \mathcal{V}_{n, e}$. By Theorem \ref{con_thresh_th}, we have that $G$ is a threshold graph. We denote its (stepwise) adjacency matrix by $A$. Also, let $x$ be the corresponding Perron vector of $G$ and let $\rho = \rho(G)$. Recall that $x_1 \ge x_2 \ge \cdots \ge x_n > 0$. Since $\mathcal{D}_{n, e}$ contains $K_{k + 1}$ as a proper subgraph, we have $\gamma > k$, which implies $\rho > k$. Besides, $\mathcal{V}_{n, e}$ contains $K_2 \lor eK_1$ as a proper subgraph, hence $\chi > k$ due to Lemma \ref{third_lemma}.
    
    Let $c$ be maximal such that $A_{c, c+1} = 1$. Since $G \not\cong \mathcal{V}_{n, e}$, we have $3 \le c \le k$. Now, for each $i \in \{2, 3, \ldots, c\}$, let $s_i$ be maximal such that $A_{i, s_i} = 1$. Note that $k + 2 \le s_2 \le e + 1$ and $\sum_{i = 2}^c (s_i - i) = e$. Let $h$ be maximal such that $A_{k + 2, h} = 1$ and observe that $2 \le h \le c$. Similarly to Proposition \ref{comparison_prop}, we have $x^\intercal (C - A) z = (\chi - \rho) x^\intercal z$, where $x^\intercal z > 0$. Let $r$ be the number of entries of $C - A$ above the main diagonal that are equal to one, i.e., $r = e + 2 - s_2 = \sum_{i = 3}^c (s_i - i)$. Therefore, $x^\intercal (C - A) z = \alpha - \beta$, where
    \[
        \alpha = x_2 (z_{s_2 + 1} + z_{s_2 + 2} + \cdots + z_{e + 2}) + z_2(x_{s_2 + 1} + x_{s_2 + 2} + \cdots + x_{e + 2}) = r(x_2 z_3 + z_2 x_n),
    \]
    while
    \begin{align*}
        \beta = \sum_{i = 3}^c \sum_{j = i + 1}^{s_i} (x_i z_j + x_j z_i) &= z_3 \sum_{i = 3}^c \sum_{j = i + 1}^{s_i} (x_i + x_j)\\
        &\le z_3 \sum_{i = 3}^c \sum_{j = i + 1}^{s_i} (x_2 + x_j) = r x_2 z_3 + z_3 \sum_{i = 3}^c \sum_{j = i + 1}^{s_i} x_j .
    \end{align*}
    If we let
    \[
        q = \frac{1}{r} \sum_{i = 3}^c \sum_{j = i + 1}^{s_i} x_j,
    \]
    then from $\rho \ge \chi$, we obtain $\beta \ge \alpha$, which leads us to $r (x_2 z_3 + q z_3) \ge r(x_2 z_3 + z_2 x_n)$, i.e.,
    \begin{equation}\label{last_aux_10}
        \frac{q}{x_n} \ge \frac{z_2}{z_3} .
    \end{equation}

    To resume, we define $q_1$ and $q_2$ as
    \[
        q_1 = \frac{1}{r} \sum_{i = 3}^h \sum_{j = i + 1}^{s_i} x_j \qquad \mbox{and} \qquad q_2 = \frac{1}{r} \sum_{i = h + 1}^c \sum_{j = i + 1}^{s_i} x_j,
    \]
    so that $q = q_1 + q_2$. Also, let $a' = \frac{\sum_{j = 2}^{k + 2} x_j}{k + 1}$. For any $i \in \{3, 4, \ldots, h\}$, we have $s_i \ge k + 2$, hence
    \[
        \frac{\sum_{j = i + 1}^{s_i} x_j}{s_i - i} \le \frac{\sum_{j = 2}^{k + 2} x_j}{k + 1} = a' .
    \]
    Therefore, we obtain
    \begin{equation}\label{last_aux_1}
        q_1 \le \frac{a'}{r} \sum_{i = 3}^h (s_i - i) .
    \end{equation}
    Since $e < \binom{k + 1}{2}$ and $x_1 \ge x_2 \ge \cdots \ge x_n > 0$, it is not difficult to conclude that
    \begin{equation}\label{last_aux_4}
        \sum_{j = 2}^{k + 2} (\rho + 1) x_j < (k + 1) \sum_{j = 1}^{k + 2} x_j,
    \end{equation}
    which implies
    \[
        (\rho + 1)a' = \frac{1}{k + 1} \sum_{j = 2}^{k + 2} (\rho + 1) x_j < \sum_{j = 1}^{k + 2} x_j = (k + 1)a' + x_1 ,
    \]
    i.e.,
    \begin{equation}\label{last_aux_2}
        (\rho - k) a' < x_1 .
    \end{equation}
    By combining \eqref{last_aux_1} and \eqref{last_aux_2}, we reach
    \begin{equation}\label{last_aux_6}
        (\rho - k) q_1 \le \frac{x_1}{r} \sum_{i = 3}^h (s_i - i) .
    \end{equation}

    We now turn our attention to $q_2$. Let $a'' = \frac{\sum_{j = 1}^{k + 2} x_j}{\rho + 1}$. Since $s_{h + 1} < k + 2$, we have
    \[
        (\rho + 1) x_i = \sum_{j = 1}^{s_i} x_j < \sum_{j = 1}^{k + 2} x_j = (\rho + 1) a''
    \]
    for any $i \in \{ h + 1, h + 2, \ldots, c \}$. Therefore,
    \begin{equation}\label{last_aux_5}
        q_2 \le \frac{a''}{r} \sum_{i = h + 1}^c (s_i - i) .
    \end{equation}
    Moreover, from \eqref{last_aux_4} we obtain $(\rho - k) \sum_{j = 1}^{k + 2} x_j < (\rho + 1) x_1$, i.e.,
    \begin{equation}\label{last_aux_3}
        (\rho - k) a'' < x_1 .
    \end{equation}
    By combining \eqref{last_aux_5} and \eqref{last_aux_3}, we conclude that
    \begin{equation}\label{last_aux_7}
        (\rho - k) q_2 \le \frac{x_1}{r} \sum_{i = h + 1}^c (s_i - i).
    \end{equation}

    From \eqref{last_aux_6} and \eqref{last_aux_7}, we get
    \[
        (\rho - k) q \le \frac{x_1}{r} \sum_{i = 3}^c (s_i - i) = x_1 = \rho x_n ,
    \]
    which means that
    \begin{equation}\label{last_aux_9}
        \frac{q}{x_n} \le \frac{\rho}{\rho - k} = 1 + \frac{k}{\rho - k} \le 1 + \frac{k}{\chi - k} = \frac{\chi}{\chi - k} .
    \end{equation}
    From \eqref{hnm_local_2} and \eqref{hnm_local_3}, we obtain
    \begin{equation}\label{last_aux_8}
        \frac{z_2}{z_3} = \frac{\chi + e}{\chi + 1} .        
    \end{equation}
    Finally, by combining \eqref{last_aux_10}, \eqref{last_aux_9} and \eqref{last_aux_8}, we have $\frac{\chi + e}{\chi + 1} \le \frac{\chi}{\chi - k}$, i.e., $\chi \le \ell_e$. We have already shown that $\ell_e > k + 1 > \rho(\mathcal{T}_1(\mathcal{V}_{n, e}))$, which means that a contradiction follows directly from Lemma \ref{poly_2_lemma}, Lemma \ref{rat_lemma} and Corollary \ref{rat_cor}.
\end{proof}

We end the paper with the following corollary of Theorem \ref{last_th}.

\begin{corollary}
    For any $e > 130$ and $n \ge e + 2 + 13 \sqrt{e}$, the graph $\mathcal{V}_{n, e}$ is the unique graph that attains the maximum spectral radius on $\mathcal{C}_{n, e}$.
\end{corollary}
\begin{proof}
    First of all, suppose that $t = 0$. In this case we have $e = \binom{k}{2}$ and $k \ge 17$, hence the result directly follows from Theorem \ref{bell_th} by trivially verifying that
    \[
        \frac{1}{2} (k + 1)(k + 6) + 7 + \frac{32}{k - 3} + \frac{16}{(k - 3)^2} < \frac{k(k - 1)}{2} + 2 + \frac{13}{\sqrt{2}} \sqrt{k(k - 1)}
    \]
    holds for any $k \ge 17$.

    In the rest of the proof, we assume that $t \ge 1$. By virtue of Theorem \ref{last_th}, it suffices to show that
    \begin{equation}\label{last_aux_11}
        e + 2 + \frac{\ell(\ell + 1)(\ell^2 - \ell - 2e)}{\ell^2 - e} < e + 2 + 13 \sqrt{e},
    \end{equation}
    where $\ell = \frac{ek}{e - k - 1}$. It is trivial to verify by computer that \eqref{last_aux_11} holds when $86 \le e \le 350$. Now, suppose that $e \ge 351$ and note that $k \ge 27$. Since $\sqrt{e} > \sqrt{\frac{k(k - 1)}{2}} > \frac{k - 1}{\sqrt{2}}$, to finalize the proof, it is enough to show that
    \[
        \frac{\ell(\ell + 1)(\ell^2 - \ell - 2e)}{(k - 1)(\ell^2 - e)} < \frac{13}{\sqrt{2}}.
    \]
    Due to $\ell = k + \frac{k^2 + k}{e - k - 1}$, we have
    \begin{equation}\label{last_aux_12}
        \ell \le k + \frac{k^2 + k}{\left(\binom{k}{2} + 1\right) - k - 1} = \frac{k^2 - k + 2}{k - 3}
    \end{equation}
    and
    \begin{equation}\label{last_aux_13}
        \ell \ge k + \frac{k^2 + k}{\left(\binom{k}{2} + k - 1\right) - k - 1} = \frac{k^3 + k^2 - 2k}{k^2 - k - 4} .
    \end{equation}
    From \eqref{last_aux_12}, we get
    \begin{equation}\label{last_aux_14}
        \ell (\ell + 1) \le \frac{(k^2 - k + 2)(k^2 - 1)}{(k-3)^2}
    \end{equation}
    and
    \begin{equation}\label{last_aux_15}
        \ell(\ell - 1) - 2e \le \frac{(k^2 - k + 2)(k^2 - 2k + 5)}{(k - 3)^2} - 2 \left( \tbinom{k}{2} + 1 \right) = \frac{4(k - 1)(k^2 - k + 2)}{(k - 3)^2},
    \end{equation}
    while \eqref{last_aux_13} leads to
    \begin{align}\label{last_aux_16}
    \begin{split}
        \ell^2 - e &\ge \frac{(k^3 + k^2 - 2k)^2}{(k^2 - k - 4)^2} - \left( \tbinom{k}{2} + k - 1 \right)\\
        &= \frac{(k-1)(k+2)(k^2 + k - 4)(k^2 + 3k + 4)}{2(k^2 - k - 4)^2} .
    \end{split}
    \end{align}
    By combining \eqref{last_aux_14}--\eqref{last_aux_16}, we conclude that
    \[
        \frac{\ell(\ell + 1)(\ell^2 - \ell - 2e)}{(k - 1)(\ell^2 - e)} \le \frac{8(k + 1)(k^2 - k - 4)^2 (k^2 - k + 2)^2}{(k - 3)^4 (k + 2) (k^2 + k - 4) (k^2 + 3k + 4)} .
    \]
    The desired result now follows by trivially observing that
    \[
        \frac{8(k + 1)(k^2 - k - 4)^2 (k^2 - k + 2)^2}{(k - 3)^4 (k + 2) (k^2 + k - 4) (k^2 + 3k + 4)} < \frac{13}{\sqrt{2}}
    \]
    holds for any $k \ge 27$.
\end{proof}

\section*{Acknowledgements}

The author is grateful to Dragan Stevanovi\'c for all the useful comments and remarks. The author is supported by the Ministry of Science, Technological Development and Innovation of the Republic of Serbia, grant number 451-03-137/2025-03/200102, and the Science Fund of the Republic of Serbia, grant \#6767, Lazy walk counts and spectral radius of threshold graphs --- LZWK.

\section*{Conflict of interest}

The author declares that he has no conflict of interest.

\end{document}